\documentclass[reqno,11pt]{amsart}
\usepackage{hyperref}
\usepackage{amsmath}
\usepackage{amsthm}
\usepackage{amssymb}
\usepackage{indentfirst}
\usepackage{amsthm} 
\newtheorem{theorem}{Theorem}[section]

\newtheorem{lemma}{Lemma}[section]

\newtheorem{proposition}{Proposition}[section]

\newcommand{\grad}{\nabla}

\newcommand{\lap}{\Delta}

\newcommand{\Ric}{\operatorname{Ric}}
\newcommand{\Ricphi}{\operatorname{Ric}_\phi}

\newcommand{\inner}[2]{\left\langle #1,#2 \right\rangle}

\usepackage{tikz}
\usetikzlibrary{calc} 

\usepackage{float}
\usepackage[top=1. 2in,bottom=1. 5in,left=1. 2in,right=1. 2in]{geometry}

\title[An Alexandrov-type theorem]{\textbf{An Alexandrov-type theorem in warped product manifolds with radial density}}

\author[Bai]{Jinchuan Bai}
\address[J.B.]{School of Mathematical Sciences\\
Xiamen University\\
361005\\
Xiamen\\
P. R.  China}
\email{baijinchuan@stu.xmu.edu.cn
}

\author[Xia]{Chao Xia}
\address[C.X.]{School of Mathematical Sciences\\
Xiamen University\\
361005\\
Xiamen\\
P. R.  China}
\email{chaoxia@xmu.edu.cn}
\thanks{This work is supported by
NSFC (Grant No.  12271449, 12526203, 12526102) and the Natural Science Foundation of
Fujian Province of China (Grant No.  2024J011008).}

\begin{document}

\maketitle
\begin{abstract}
In this paper, 
we establish a Heintze--Karcher inequality for closed embedded
hypersurfaces in a class of warped product manifolds endowed with radial density. As a consequence, we prove Alexandrov-type theorem for constant weighted mean curvature hypersurfaces in such spaces.
In particular, we prove that a closed
embedded $\lambda$-self-expander in the Euclidean space must be a round sphere centered at the origin.
\end{abstract}

\section{Introduction}

 Alexandrov's celebrated soap-bubble theorem asserts
that every closed, connected, embedded hypersurface in Euclidean space with
constant mean curvature must be a round sphere \cite{Alexandrov1956}.
In addition to Alexandrov's original moving-plane method, several integral
approaches have proved to be particularly effective. Reilly gave a proof based
on his integral formula \cite{Reilly1977}, whereas Ros combined
Heintze--Karcher inequalities with Minkowski-type identities to establish
rigidity results for hypersurfaces with constant higher-order mean curvatures
\cite{Ros1987}. Along this direction, a celebrated result is on Alexandrov's theorem on warped product manifolds by Brendle \cite{Brendle2013}. 

In recent decades, analogous rigidity questions have been studied extensively
in the setting of manifolds with density. Let $(M^{m+1},\bar g)$ be a
Riemannian manifold endowed with the smooth weighted measure
\[
d\mu_\phi=e^{-\phi}\,d\mu_{\bar g},
\]
where $\phi\in C^\infty(M)$ is called the potential function. If
$\Omega\subset M$ is a smooth domain with boundary
$\Sigma=\partial\Omega$, its weighted volume and the weighted area of
$\Sigma$ are defined, respectively, by
\[
\operatorname{Vol}_\phi(\Omega)
=
\int_\Omega e^{-\phi}\,d\mu_{\bar g},
\qquad
\operatorname{Area}_\phi(\Sigma)
=
\int_\Sigma e^{-\phi}\,d\sigma.
\]Manifolds
with density and their associated isoperimetric problems have been
extensively investigated; see, for example, \cite{Morgan2005,
RosalesCaneteBayleMorgan2008}.
A particularly important model is Euclidean space equipped with a radial
log-convex density.
 Rosales, Ca\~nete, Bayle, and Morgan
investigated the corresponding weighted isoperimetric problem and showed,
among other results, that centered balls are stable in the radial log-convex
setting \cite{RosalesCaneteBayleMorgan2008}. Their work led to the
well-known Log-Convex Density Conjecture, commonly attributed to Brakke,
which asserts that balls centered at the origin minimize weighted perimeter
among all sets with a prescribed weighted volume.
Important partial results toward this conjecture were obtained by
Kolesnikov--Zhdanov \cite{KolesnikovZhdanov2011} and Figalli--Maggi
\cite{FigalliMaggi2013}. In particular, Figalli and Maggi established the
conjecture in several significant regimes, including the case of small
weighted volumes and densities sufficiently close to the quadratic Gaussian
model. The conjecture was ultimately resolved in full by Chambers
\cite{Chambers2019}, who proved that, for every smooth radial log-convex
density, the weighted isoperimetric regions are precisely the balls centered
at the origin.

The natural curvature quantity arising in the first variation of weighted
area is the weighted mean curvature
\[
H_\phi
:=
H-\langle\overline\nabla\phi,\nu\rangle,
\]
where $\nu$ is a choice of unit normal and $H$ denotes the unnormalized mean
curvature. The first variation formula shows that the critical points of the
weighted area functional under weighted-volume-preserving variations are
precisely the hypersurfaces satisfying
\begin{equation}\label{eq:CWMC}
H_\phi=\lambda
\end{equation}
for some constant $\lambda\in\mathbb R$. Such hypersurfaces are referred to
as constant weighted mean curvature, or $\phi$-CMC, hypersurfaces. 
The most fundamental and interesting example of a strictly
log-convex density is the  density
\begin{equation}\label{eq:expander-density}
\rho(x)=e^{|x|^2/4},
\end{equation}
which we call it anti-Gaussian density.
In this case, the constant weighted mean curvature hypersurfaces are given by
\begin{equation}\label{eq:lambda-self-expander}
H+\frac12\langle x,\nu\rangle=\lambda.
\end{equation}
In reference, it is usually called
\emph{$\lambda$-self-expanders}.

The anti-Gaussian setting
stands in contrast to the  Gaussian density setting with  $\rho(x)=e^{-|x|^2/4}$, in which
constant weighted mean curvature hypersurfaces give rise to the theory of
$\lambda$-hypersurfaces which satisfying 
\begin{equation}\label{eq:lambda-self-shrinker}
H-\frac12\langle x,\nu\rangle=\lambda.
\end{equation}
see Cheng--Wei \cite{ChengWei2018}. In particular,
the case $\lambda=0$ in the shrinking setting corresponds to the theory of
self-shrinkers, which plays a central role in the analysis of singularities
in mean curvature flow, see for example,
Colding--Minicozzi \cite{ColdingMinicozzi2012}.

We return to the $\lambda$-self-expanders. When $\lambda=0$, equation \eqref{eq:lambda-self-expander} reduces, up to the
choice of orientation and sign convention for the mean curvature, to the
self-expander equation
\[
H+\frac12\langle x,\nu\rangle=0.
\]
Self-expanders generate self-similarly expanding solutions to mean curvature
flow and play an important role in the study of flows emerging from conical
singularities, as well as in the analysis of the large-time behavior of
geometric evolutions. The existence and asymptotic geometry of self-expanders
associated with cones were studied by Ding \cite{Ding2020}. Deruelle and
Schulze established important uniqueness results for asymptotically conical
self-expanders \cite{DeruelleSchulze2020}, while Bernstein and Wang developed
a systematic theory for the moduli space of asymptotically conical
self-expanders \cite{BernsteinWang2021}.

$\lambda$-self-expanders has also received considerable
attention. Ancari and Cheng \cite{AncariCheng2023} investigated rigidity
properties of $\lambda$-self-expanders and obtained several characterization
results under additional geometric assumptions. 
Our main result is the following Alexadrov-type theorem about $\lambda$-self-expanders. 
\begin{theorem}\label{thm1. 1}
Let $\Sigma^m\subset\mathbb R^{m+1}$ be a smooth, closed,
embedded $\lambda$-self-expander. Then $\Sigma$ is a round sphere
centered at the origin. 
\end{theorem}

In fact, Theorem \ref{thm1. 1} is a special case for the following more general result concerning Alexandrov-type theorem in 
certain warped product manifolds with radial density. 
More precisely, we consider 
\[
(M^{m+1},\bar g, \mu)
=
\bigl(N^m\times[0,\bar r),\,dr^2+\lambda(r)^2g_N, e^{-\phi(r)}\mu_{\bar{g}}\bigr)
\]
where $N$ is a closed $m$-manifold satisfying $\operatorname{Ric}_{N}\ge (m-1)\kappa_N g_N$ for some $\kappa_N \in \mathbb R$ and $\lambda$ is a smooth non-negative function on $[0,\bar r)$.
We impose the following  conditions on the warping function $\lambda$ and the potential $\phi$. 
\medskip
\begin{align*}
\hbox{ either }\textnormal{(C1)}\quad &
\lambda(r)>0, \quad\text{for }r\in[0,\bar r), \\
\hbox{ or }\textnormal{(C1')}\quad & \lambda(r)=r\eta(r^2), \hbox{ where }\eta\hbox{ is a smooth positive function with }\eta(0)=1;\\
\textnormal{(C2)}\quad &
\lambda'(r)>0
\quad\text{for }r\in[0,\bar r);\\
\textnormal{(C3)}\quad &
\phi'(r)\le0
\text{ for }r\in[0,\bar r);\\
\textnormal{(C4)}\quad &
\mathcal C_{\lambda,\phi}\ge0, \hbox{ where } C_{\lambda,\phi}\hbox{ is defined in \eqref{C-lambda} in Section \ref{subsec:2}};\\
\textnormal{(C5)} \quad &\phi'(0)=0, \hbox{ in case }\lambda(r)=r\eta(r^2).
\end{align*}
\noindent
In the case (C1), $M$ is a manifold with boundary $\partial M=N\times \{0\}$ and in the case (C1'), $M$ is smooth at the origin given by $\{r=0\}$.
We remark that $(C4)$ is 
 to ensure  the weighted sub-static condition
 \begin{equation}\label{substatic}
\overline{\lap}_\phi V \bar g
-\overline{\grad}^2V
+V\operatorname{Ric}_{\phi}
+V\phi'\frac{\lambda'}{\lambda}\bar g
\ge0,
\end{equation}
where $V=m\lambda'(r)-\phi'(r)\lambda(r)$.
see Lemma \ref{lem2. 4}. 

\begin{theorem}
\label{cmc-rigidity}
Assume that \(\phi(r)\) and \(\lambda(r)\) satisfy (C1)-(C4) or (C1') and (C2)-(C5). Let $\Sigma\subset M^{m+1}$ be a smooth closed embedded $\phi$-CMC hypersurface, i. e. , $H_\phi$ is constant on $\Sigma$. Then one of the following alternatives holds:
\begin{itemize}
\item[(1)] $\Sigma$ is a coordinate slice $N\times \{r_0\}$ for some $r_0\in (0, \bar r)$; or
\item[(2)] $\Sigma$ is a totally umbilical hypersurface contained in the region where $\phi$ is constant. 
\end{itemize}
\end{theorem}
The proof of Theorem \ref{cmc-rigidity} is based on the following Heintze-Karcher-type inequalities.
\begin{theorem}\label{thm1. 2}
Assume that \(\phi(r)\) and \(\lambda(r)\) satisfy (C1') and (C2)-(C5). 
Let  \(\Sigma\subset N\times(0,\bar r)\) be a smooth closed
hypersurface satisfying $H_\phi>0$.
Then
\begin{equation}\label{hk1}
\int_\Sigma\frac{m\lambda'(r)-\phi'(r)\lambda(r)}{H_\phi}\,d\sigma_\phi
\ge
\int_\Omega [(m+1)\lambda'(r)-\phi'(r)\lambda(r)]\,d\mu_\phi. 
\end{equation}
Furthermore, equality in \eqref{hk1} holds if and only if one of the following alternatives holds:
\begin{enumerate}
\item $\Sigma$ is a coordinate slice $N\times\{r_0\}$ for some $r_0\in (0, \bar r)$; or
\item $\Sigma$ is a totally umbilical hypersurface contained in the region where $\phi$ is constant. 
\end{enumerate}
\end{theorem}
\begin{theorem}\label{thm1. 2'}
Assume that \(\phi(r)\) and \(\lambda(r)\) satisfy (C1)-(C4). 
Let  \(\Sigma\subset N\times(0,\bar r)\) be a smooth closed
hypersurface satisfying $H_\phi>0$.
\begin{enumerate}
\item If $\Sigma$ is null-homologous, then \eqref{hk1} holds.
\item If $\Sigma$ is homologous to $N_0=N\times \{0\}$,
then
\begin{equation}\label{hk2}
\int_\Sigma\frac{m\lambda'(r)-\phi'(r)\lambda(r)}{H_\phi}\,d\sigma_\phi
\ge
\int_\Omega [(m+1)\lambda'(r)-\phi'(r)\lambda(r)]\,d\mu_\phi
+
\lambda(0)\int_{N_0 \cap \partial \Omega}d\sigma_\phi. 
\end{equation}
Furthermore, equality in \eqref{hk2} holds if and only if one of the following alternatives holds:

\begin{enumerate}
\item $\Sigma$ is a coordinate slice $N\times\{r_0\}$ for some $r_0\in (0, \bar r)$; or
\item $\Sigma$ is a totally umbilical hypersurface contained in the region where $\phi$ is constant. 
\end{enumerate}
\end{enumerate}
\end{theorem}
Note that when $\phi$ is a constant, then Theorems \ref{cmc-rigidity}, \ref{thm1. 2} and \ref{thm1. 2'} are the celebrated result by Brendle \cite{Brendle2013}.
The Heintze--Karcher inequality originates from the comparison theorem of
Heintze and Karcher \cite{HeintzeKarcher1978}. A major extension has been made by
Brendle \cite{Brendle2013}, who established Heintze--Karcher-type
inequality for substatic warped product manifolds and applied it to Alexandrov-type theorem in certain spaces. The key innovation in Brendle's proof is a normal flow constructed with respect to a
conformal metric.
Li and the second-named author \cite{LiXia2019} subsequently established a generalized Reilly-type formula, obtaining Heintze--Karcher and Minkowski-type inequalities in general substatic manifolds, which leads to a new proof of Brendle's result. The equality case of the sub-static
Heintze--Karcher inequality and its rigidity consequences were further
investigated by Borghini, Fogagnolo, and Pinamonti
\cite{BorghiniFogagnoloPinamonti2024}. 
We also mention a work by Li, Wei and Xu \cite{LiWeiXu2025} on new Heintze--Karcher-type inequality in
sub-static warped product manifolds. 
Brendle's idea has been also applied to hypersurfaces with free boundary or capillary boundary, see the work by Jia, Wang, Zhang and the second-named author \cite{MR4562813,MR4817296,MR4805880}.

 The idea of our proof of Theorems \ref{thm1. 2} and \ref{thm1. 2'} are based on Brendle's idea~\cite{Brendle2013}.
The key observation is that, the quantity $\int_\Sigma\frac{V}{H_\phi}\,d\sigma_\phi$ satisfies certain monotonicity property for the normal flow with respect to the conformal metric $\widehat{g}=V^{-2}\overline{g}$, when $(M^{m+1},\bar g, \mu)$ satisfies the weighted sub-static conditon \eqref{substatic}.




The rest of the paper is organized as follows.  In Section \ref{subsec:2}, we develop the weighted
geometry of warped products and establish a weighted sub-static inequality
controlled by the compatibility quantity $C_{\lambda,\phi}$. 
In Section \ref{subsec:3}, we study the inward normal flow associated with the conformal metric
$\widehat{g}=V^{-2}\overline{g}$
and derive an evolution inequality for $H_\phi/V$, and use a weighted
inequality to obtain the required monotonicity. 
In Section \ref{subsec:4}, we integrate this inequality along the flow and apply the
weighted coarea formula. This yields the weighted Heintze--Karcher inequality in
Theorems \ref{thm1. 2} and \ref{thm1. 2'}. 
Finally, in the anti-Gaussian setting, a weighted Minkowski
identity for $\lambda$-self-expanders implies equality in associated Heintze-Karcher inequality. 
The rigidity statement then shows that the hypersurface is totally umbilical,
and the $\lambda$-self-expander equation forces its center to be the origin. 
This proves Theorem \ref{thm1. 1}.
In Appendix A, we state the corresponding result in space forms and indicate interesting weight satisfying the required conditions.

\

    \noindent{\bf Addendum.} We note that  Florian Johne and Lauro Silini \cite{JohneSilini2026} have simultaneously and independently obtained  similar results via a related technique.

\

\section{Warped‑product manifolds with radial density}\label{subsec:2}

Let $(N^m,g_N)$ be a closed Riemannian manifold, and consider the warped
product manifold
\[
(M^{m+1},\bar g)
=
\bigl(N\times[0,\bar r),\,dr^2+\lambda(r)^2g_N\bigr),
\]
where the warping function $\lambda$ satisfies
\begin{equation}\label{lambda}
\lambda\in C^\infty([0,\bar r)),
\qquad
\lambda(r)>0,
\qquad
\lambda'(r)>0
\quad\text{for all }r\in[0,\bar r).
\end{equation}
Unless otherwise specified, $\langle\cdot,\cdot\rangle$ denotes the inner
product with respect to the ambient metric $\bar g$. Its restriction to
$T\Sigma$ coincides with the metric induced on $\Sigma$.

Assume that $N$ satisfies the Ricci curvature lower bound
\begin{equation}\label{Ric}
\operatorname{Ric}_{N}\ge (m-1)\kappa_N g_N
\end{equation}
for some constant $\kappa_N\in\mathbb R$. Define
\begin{equation}\label{T}
\mathcal T_{\lambda,\kappa_N}(r)
:=
\frac{
(m-1)\kappa_N-\lambda(r)\lambda''(r)
-(m-1)(\lambda'(r))^2
}{\lambda(r)^2}.
\end{equation}
When $N=\mathbb S^m$ and $\kappa_N=1$, this quantity becomes
\[
\mathcal T_{\lambda,1}
=
\frac{(m-1)\bigl(1-(\lambda')^2\bigr)-\lambda\lambda''}{\lambda^2},
\]
which is precisely the curvature quantity arising in the rotationally
symmetric setting.

Let $\phi=\phi(r)$ be a smooth radial potential. We equip $M$ and a
hypersurface $\Sigma\subset M$ with the weighted measures
\[
d\mu_\phi:=e^{-\phi}d\mu_{\bar g},
\qquad
d\sigma_\phi:=e^{-\phi}d\sigma,
\]
respectively. Throughout, we assume that
\begin{equation}\label{phi}
\phi'(r)\le 0
\qquad\text{for all }r\in[0,\bar r).
\end{equation}
If $r=0$ is a pole of the warped product, one additionally requires
$\phi'(0)=0$.

The ambient weighted Laplacian, or Bakry--Émery operator, and the weighted
Laplacian on $\Sigma$ are defined, respectively, by
\[
\overline{\Delta}_\phi f
=
\overline{\Delta}f
-
\bigl\langle\overline{\nabla}\phi,\overline{\nabla}f\bigr\rangle,
\qquad
\Delta_\phi f
=
\Delta f-\langle\nabla\phi,\nabla f\rangle.
\]
Here $\overline{\nabla}$ and $\overline{\Delta}$ denote the ambient gradient
and Laplacian on $M$, whereas $\nabla$ and $\Delta$ denote the tangential
gradient and Laplace--Beltrami operator on $\Sigma$.

For every smooth vector field $X$ on $M$, its weighted divergence is defined
by
\[
\operatorname{div}_{\phi}X
=
\operatorname{div}X-\langle\overline{\nabla}\phi,X\rangle.
\]
The weighted divergence theorem then takes the form
\begin{equation}\label{eq:weighted-divergence}
\int_{\Omega}\operatorname{div}_{\phi}X\,d\mu_\phi
=
\int_{\partial\Omega}\langle X,\nu\rangle\,d\sigma_\phi.
\end{equation}

The weighted mean curvature of $\Sigma$ with respect to the potential $\phi$
is given by
\[
H_\phi
=
H-\langle\overline{\nabla}\phi,\nu\rangle,
\]
and the Bakry--Émery Ricci tensor of the ambient weighted manifold is
\[
\operatorname{Ric}_\phi
=
\overline{\operatorname{Ric}}+\overline{\nabla}^2\phi,
\]
where $\overline{\operatorname{Ric}}$ denotes the Ricci tensor of
$(M^{m+1},\bar g)$. We say that $\Sigma$ is a
\emph{$\phi$-constant-mean-curvature} (or \emph{$\phi$-CMC}) hypersurface if
$H_\phi$ is constant on $\Sigma$.

Define
\begin{equation}\label{V}
V(r):=m\lambda'(r)-\phi'(r)\lambda(r)
\end{equation}
and
\begin{equation}\label{W}
W(r):=(m+1)\lambda'(r)-\phi'(r)\lambda(r).
\end{equation}
In view of \eqref{lambda} and \eqref{phi}, we have
\[
V>0,
\qquad
W>0
\quad\text{on }[0,\bar r).
\]

Finally, define the compatibility quantity
\begin{equation}\label{C-lambda}
\mathcal C_{\lambda,\phi}
:=
V''
+
\left((m-1)\frac{\lambda'}{\lambda}-\phi'\right)V'
+
\mathcal T_{\lambda,\kappa_N}V
+
2\phi'\frac{\lambda'}{\lambda}V.
\end{equation}
Our main structural assumption is
\begin{equation}\label{C}
\mathcal C_{\lambda,\phi}\ge0
\qquad\text{on }[0,\bar r).
\end{equation}

With the preceding notation and structural assumptions in place, we first
establish several technical lemmas that will be used repeatedly in the proofs
of the main results.

\begin{lemma}\label{lem:weighted-divergence-X}
Let
\[
X=\lambda(r)\partial_r.
\]
Then
\[
\operatorname{div}_{\phi}X
:=
\operatorname{div}_{\bar g}X
-
\langle\overline{\nabla}\phi,X\rangle
=
W(r).
\]
\end{lemma}

\begin{proof}
Let $Y$ be any vector field tangent to the slice
\[
N_r=N\times\{r\}.
\]
The connection formulas for the warped product metric yield
\[
\overline{\nabla}_{\partial_r}\partial_r=0,
\qquad
\overline{\nabla}_{Y}\partial_r
=
\frac{\lambda'}{\lambda}Y.
\]
Therefore,
\[
\overline{\nabla}_{\partial_r}
\bigl(\lambda\partial_r\bigr)
=
\lambda'\partial_r,
\]
and
\[
\overline{\nabla}_{Y}
\bigl(\lambda\partial_r\bigr)
=
\lambda\,\overline{\nabla}_{Y}\partial_r
=
\lambda'Y.
\]
It follows that
\[
\overline{\nabla}X=\lambda'\bar g.
\]
Taking the trace with respect to $\bar g$, we obtain
\[
\operatorname{div}_{\bar g}X=(m+1)\lambda'(r).
\]
Since $\phi=\phi(r)$, we also have
\[
\overline{\nabla}\phi=\phi'(r)\partial_r,
\]
and hence
\[
\langle\overline{\nabla}\phi,X\rangle
=
\phi'(r)\lambda(r).
\]
Consequently,
\[
\operatorname{div}_{\phi}X
=
(m+1)\lambda'(r)-\phi'(r)\lambda(r)
=
W(r).
\]
\end{proof}

\begin{lemma}[Weighted Minkowski identity]\label{lem:weighted-Minkowski}
Let
\[
L(r):=\int_0^r\lambda(\tau)\,d\tau,
\qquad
X:=\overline{\nabla}L=\lambda(r)\partial_r.
\]
Then
\[
\operatorname{div}_{\Sigma,\phi}X^\top
=
V-H_\phi\langle X,\nu\rangle,
\]
where
\[
\operatorname{div}_{\Sigma,\phi}Y
:=
\operatorname{div}_\Sigma Y-\langle\nabla\phi,Y\rangle
\]
for every tangent vector field $Y$ on $\Sigma$. Consequently,
\[
\int_\Sigma V\,d\sigma_\phi
=
\int_\Sigma H_\phi\langle X,\nu\rangle\,d\sigma_\phi.
\]
\end{lemma}

\begin{proof}
Let $\{e_1,\ldots,e_m\}$ be a local orthonormal frame on $\Sigma$. Since
\[
X=\lambda\partial_r
\]
is a conformal vector field satisfying
\[
\overline{\nabla}X=\lambda'\bar g,
\]
we have
\[
\langle\overline{\nabla}_{e_i}X,e_i\rangle=\lambda'
\qquad\text{for }i=1,\ldots,m.
\]

Decompose $X$ along $\Sigma$ as
\[
X=X^\top+\langle X,\nu\rangle\nu.
\]
Then
\[
\begin{aligned}
\operatorname{div}_\Sigma X^\top
&=
\sum_{i=1}^m
\left\langle\overline{\nabla}_{e_i}X^\top,e_i\right\rangle\\
&=
\sum_{i=1}^m
\left\langle
\overline{\nabla}_{e_i}
\bigl(X-\langle X,\nu\rangle\nu\bigr),
e_i
\right\rangle\\
&=
\sum_{i=1}^m
\langle\overline{\nabla}_{e_i}X,e_i\rangle
-
\langle X,\nu\rangle
\sum_{i=1}^m
\langle\overline{\nabla}_{e_i}\nu,e_i\rangle\\
&=
m\lambda'-H\langle X,\nu\rangle.
\end{aligned}
\]
Here we have used the convention
\[
H=\sum_{i=1}^m\langle\overline{\nabla}_{e_i}\nu,e_i\rangle.
\]

Since $\phi=\phi(r)$, we have
\[
\overline{\nabla}\phi=\phi'(r)\partial_r.
\]
Moreover, using
\[
X^\top=X-\langle X,\nu\rangle\nu,
\]
we obtain
\[
\begin{aligned}
\langle\nabla\phi,X^\top\rangle
&=
\langle\overline{\nabla}\phi,X^\top\rangle\\
&=
\langle\overline{\nabla}\phi,X\rangle
-
\langle\overline{\nabla}\phi,\nu\rangle\langle X,\nu\rangle\\
&=
\phi'\lambda
-
\langle\overline{\nabla}\phi,\nu\rangle\langle X,\nu\rangle.
\end{aligned}
\]
Therefore,
\[
\begin{aligned}
\operatorname{div}_{\Sigma,\phi}X^\top
&=
\operatorname{div}_\Sigma X^\top
-
\langle\nabla\phi,X^\top\rangle\\
&=
m\lambda'-H\langle X,\nu\rangle
-\phi'\lambda
+\langle\overline{\nabla}\phi,\nu\rangle\langle X,\nu\rangle\\
&=
m\lambda'-\phi'\lambda
-
\bigl(H-\langle\overline{\nabla}\phi,\nu\rangle\bigr)
\langle X,\nu\rangle\\
&=
V-H_\phi\langle X,\nu\rangle.
\end{aligned}
\]

Finally, since $\Sigma$ is closed, the weighted divergence theorem on
$\Sigma$ implies that
\[
\int_\Sigma \operatorname{div}_{\Sigma,\phi}X^\top\,d\sigma_\phi=0.
\]
Hence,
\[
0
=
\int_\Sigma
\left(V-H_\phi\langle X,\nu\rangle\right)
\,d\sigma_\phi,
\]
which proves that
\[
\int_\Sigma V\,d\sigma_\phi
=
\int_\Sigma H_\phi\langle X,\nu\rangle\,d\sigma_\phi.
\]
\end{proof}

\begin{lemma}\label{lem:weighted-mean-curvature-estimate}
Assume that \(\phi(r)\) and \(\lambda(r)\) satisfy (C1)-(C4) or (C1') and (C2)-(C5).Let $\Sigma\subset M$ be a hypersurface with unit normal $\nu$, second
fundamental form $A$, and mean curvature $H$. Then, pointwise on $\Sigma$,
\begin{equation}\label{eq:Hphi-estimate}
H_\phi^2
\le
\left(
|A|^2-\phi'\frac{\lambda'}{\lambda}
\right)
\frac{V}{\lambda'}.
\end{equation}
\end{lemma}

\begin{proof}
Define
\[
b(r):=-\frac{\phi'(r)}{\lambda(r)}\ge0.
\]
Since
\[
\overline{\nabla}\phi=\phi'(r)\partial_r
\qquad\text{and}\qquad
X=\lambda(r)\partial_r,
\]
we have
\[
\langle\overline{\nabla}\phi,\nu\rangle
=
\frac{\phi'}{\lambda}\langle X,\nu\rangle.
\]
Consequently,
\[
H_\phi
=
H-\langle\overline{\nabla}\phi,\nu\rangle
=
H-\frac{\phi'}{\lambda}\langle X,\nu\rangle
=
H+b\langle X,\nu\rangle.
\]

Let $\kappa_1,\ldots,\kappa_m$ be the principal curvatures of $\Sigma$.
Then
\[
H=\sum_{i=1}^m\kappa_i,
\qquad
|A|^2=\sum_{i=1}^m\kappa_i^2.
\]
Applying the Cauchy--Schwarz inequality to the vectors
\[
\bigl(\kappa_1,\ldots,\kappa_m,\sqrt{b\lambda'}\bigr)
\]
and
\[
\left(
1,\ldots,1,
\sqrt{\frac{b}{\lambda'}}\langle X,\nu\rangle
\right),
\]
we obtain
\[
H_\phi^2=
\bigl(H+b\langle X,\nu\rangle\bigr)^2
\le
\bigl(|A|^2+b\lambda'\bigr)
\left(
m+\frac{b}{\lambda'}\langle X,\nu\rangle^2
\right).
\]
Under \emph{(C1')} and \emph{(C5)}, the function
$
b=-\frac{\phi'}{\lambda}
$
and
$
\sqrt{\frac{b}{\lambda'}}
$
is well-defined at $r=0$. Since $\langle X,\nu\rangle^2\le |X|^2$ and
\[
|A|^2+b\lambda'\ge0,
\]
it follows that
\[
\begin{aligned}
H_\phi^2
&\le
\bigl(|A|^2+b\lambda'\bigr)
\left(
m+\frac{b}{\lambda'}|X|^2_{\bar{g}}
\right).
\end{aligned}
\]
Using
\[
b\lambda'
=
-\phi'\frac{\lambda'}{\lambda},
\qquad
|X|^2_{\bar{g}}=\lambda^2,
\]
we find that
\[
\begin{aligned}
m+\frac{b}{\lambda'}|X|^2_{\bar{g}}
=
m-\frac{\phi'}{\lambda\lambda'}\lambda^2
=\frac{V}{\lambda'}.
\end{aligned}
\]
Therefore,
\[
H_\phi^2
\le
\left(
|A|^2-\phi'\frac{\lambda'}{\lambda}
\right)
\frac{V}{\lambda'},
\]
which completes the proof. 
\end{proof}

\begin{lemma}
Let \(\nu\) be a unit vector in \(TM\), and  $\operatorname{Ric}_{N}\ge (m-1)\kappa_N g_N$. Then
\[
\operatorname{Ric}_{\phi}(\nu,\nu)
\ge
-\langle\nu,\partial_r\rangle^2\,m\frac{\lambda''}{\lambda}
+
(1-\langle\nu,\partial_r\rangle^2)\mathcal T_{\lambda,\kappa_N}
+
\langle\nu,\partial_r\rangle^2\phi''
+
(1-\langle\nu,\partial_r\rangle^2)\frac{\lambda'}{\lambda}\phi'. 
\]
\end{lemma}

\begin{proof}
The Ricci tensor of the warped product satisfies
\[
\operatorname{Ric}_{\bar g}(\partial_r,\partial_r)
=
-m\frac{\lambda''}{\lambda}. 
\]
If \(E\) is a unit vector tangent to \(N_r\), then
\[
\operatorname{Ric}_{\bar g}(E,E)
=
\frac{1}{\lambda^2}\operatorname{Ric}_N(\lambda E,\lambda E)
-\frac{\lambda \lambda ''+(m-1)(\lambda')^2}{\lambda^2}. 
\]
Using
\[
\operatorname{Ric}_N\ge(m-1)\kappa_N g_N,
\]
we obtain
\[
\operatorname{Ric}_{\bar g}(E,E)
\ge
\frac{(m-1)\kappa_N-\lambda \lambda ''-(m-1)(\lambda')^2}{\lambda^2}
=
\mathcal T_{\lambda,\kappa_N}. 
\]

Write
\[
\nu=\sqrt {\langle\nu,\partial_r\rangle^2}\,\partial_r+\sqrt{1-\langle\nu,\partial_r\rangle^2}\,E,
\]
where \(E\) is tangent to \(N_r\) and \(|E|=1\). Then
\[
\operatorname{Ric}_{\bar g}(\nu,\nu)
\ge
-\langle\nu,\partial_r\rangle^2\,m\frac{\lambda''}{\lambda}
+
(1-\langle\nu,\partial_r\rangle^2)\mathcal T_{\lambda,\kappa_N}. 
\]
Since
\[
\overline{\grad}^2\phi(\partial_r,\partial_r)=\phi'',
\qquad
\overline{\grad}^2\phi(E,E)=\phi'\frac{\lambda'}{\lambda}=\phi'\frac{\lambda'}{\lambda},
\]
we have
\[
\overline{\grad}^2\phi(\nu,\nu)
=
\langle\nu,\partial_r\rangle^2\phi''+(1-\langle\nu,\partial_r\rangle^2)\phi'\frac{\lambda'}{\lambda}. 
\]
The conclusion follows from
\[
\operatorname{Ric}_{\phi}
=
\operatorname{Ric}_{\bar g}+\overline{\grad}^2\phi. 
\]
\end{proof}

\begin{lemma}\label{lem2. 4}
For every unit vector \(\nu\in TM\), 
$\operatorname{Ric}_{N}\ge (m-1)\kappa_N g_N$, one has
\begin{align}
\overline{\lap}_\phi V
-\overline{\grad}^2V(\nu,\nu)
+V\operatorname{Ric}_{\phi}(\nu,\nu)
+V\phi'\frac{\lambda'}{\lambda}
\nonumber  \ge
(1-\langle\nu,\partial_r\rangle^2)\mathcal C_{\lambda,\phi}. 
\label{eq:weighted-substatic}
\end{align}
In particular, if
$\mathcal C_{\lambda,\phi}\ge0$,
then
\[
\overline{\lap}_\phi V
-\overline{\grad}^2V(\nu,\nu)
+V\operatorname{Ric}_{\phi}(\nu,\nu)
+V\phi'\frac{\lambda'}{\lambda}
\ge0. 
\]
\end{lemma}

\begin{proof}
Since \(V=V(r)\), we have
\[
\overline{\lap}_\phi V
=
V''
+
(m\frac{\lambda'}{\lambda}-\phi')V'. 
\]
Moreover,
\[
\overline{\grad}^2V(\partial_r,\partial_r)=V'',
\]
and, for every unit vector \(E\) tangent to \(N_r\),
\[
\overline{\grad}^2V(E,E)=\frac{\lambda'}{\lambda}V'. 
\]
Therefore,
\[
\overline{\grad}^2V(\nu,\nu)
=
\langle\nu,\partial_r\rangle^2V''+(1-\langle\nu,\partial_r\rangle^2)\frac{\lambda'}{\lambda}V'. 
\]

Using the preceding Ricci curvature estimate, we obtain
\begin{align*}
&
\overline{\lap}_\phi V
-\overline{\grad}^2V(\nu,\nu)
+V\operatorname{Ric}_{\phi}(\nu,\nu)
+V\phi'\frac{\lambda'}{\lambda}
\\
&\ge
(1-\langle\nu,\partial_r\rangle^2)V''
+
\bigl((m-1+\langle\nu,\partial_r\rangle^2)\frac{\lambda'}{\lambda}-\phi'\bigr)V'
\\
&\quad
+
V
\left[
-\langle\nu,\partial_r\rangle^2\,m\frac{\lambda''}{\lambda}
+
(1-\langle\nu,\partial_r\rangle^2)\mathcal T_{\lambda,\kappa_N}
+
\langle\nu,\partial_r\rangle^2\phi''
+
(2-\langle\nu,\partial_r\rangle^2)\phi'\frac{\lambda'}{\lambda}
\right]. 
\end{align*}
We separate the right-hand side into a \((1-\langle\nu,\partial_r\rangle^2)\)-part and an
\(\langle\nu,\partial_r\rangle^2\)-part:
\begin{align*}
&
\overline{\lap}_\phi V
-\overline{\grad}^2V(\nu,\nu)
+V\operatorname{Ric}_{\phi}(\nu,\nu)
+V\phi'\frac{\lambda'}{\lambda}
\\
&\ge
(1-\langle\nu,\partial_r\rangle^2)
\left[
V''
+
\bigl((m-1)\frac{\lambda'}{\lambda}-\phi'\bigr)V'
+
\mathcal T_{\lambda,\kappa_N}V
+
2\phi'\frac{\lambda'}{\lambda}V
\right]
\\
&\quad
+
\langle\nu,\partial_r\rangle^2
\left[
(m\frac{\lambda'}{\lambda}-\phi')V'
+
\left(
-m\frac{\lambda''}{\lambda}
+\phi''
+\phi'\frac{\lambda'}{\lambda}
\right)V
\right]. 
\end{align*}
Indeed, from
\[
V=m\lambda'-\phi'\lambda,
\]
we obtain
\[
m\frac{\lambda'}{\lambda}-\phi'
=
\frac{m\lambda'-\phi'\lambda}{\lambda}
=
\frac{V}{\lambda}. 
\]
On the other hand,
\begin{align*}
V'=
m\lambda''-\phi''\lambda-\phi'\lambda'
=
-\lambda\left(
-m\frac{\lambda''}{\lambda}
+\phi''
+\phi'\frac{\lambda'}{\lambda}
\right)
\end{align*}
Consequently,
\begin{align*}
&(m\frac{\lambda'}{\lambda}-\phi')V'
+
\left(
-m\frac{\lambda''}{\lambda}
+\phi''
+\phi'\frac{\lambda'}{\lambda}
\right)V
\\
&=
\frac{V}{\lambda}
\left[
-\lambda\left(
-m\frac{\lambda''}{\lambda}
+\phi''
+\phi'\frac{\lambda'}{\lambda}
\right)
\right]
+
\left(
-m\frac{\lambda''}{\lambda}
+\phi''
+\phi'\frac{\lambda'}{\lambda}
\right)V
\\
&=0. 
\end{align*}
Thus,
\[
\overline{\lap}_\phi V
-\overline{\grad}^2V(\nu,\nu)
+V\operatorname{Ric}_{\phi}(\nu,\nu)
+V\phi'\frac{\lambda'}{\lambda}
\ge
(1-\langle\nu,\partial_r\rangle^2)\mathcal C_{\lambda,\phi},
\]
which completes the proof. 
\end{proof}

\section{Weighted Normal Flows under Conformal Metric $\hat{g}=V^{-2}\bar{g}$}\label{subsec:3}

In this section, we investigate the inward normal flow induced by a suitable conformal metric. Recall that for the radial density $\phi=\phi(r)$ with $r=|x|$, we define
\[
V(r):=m\lambda'(r)-\phi'(r)\lambda(r),\quad W(r):=(m+1)\lambda'(r)-\phi'(r)\lambda(r). 
\]
Under assumption (\ref{lambda}) and (\ref{phi}), we have
\[
V(r)\geq m \inf_{x \in \Sigma}{\lambda'(|x|)}>0. 
\]
Hence, $\hat{g}=V^{-2}g$ induces a smooth Riemannian metric on $\overline{\Omega}$. 

The inward normal flow with respect to $\hat{g}$ corresponds to the Euclidean normal flow with speed $V$:
\[
\partial_t F = -V\nu. 
\]
We then derive the evolution equations for the weighted mean curvature $H_\phi$. 

\begin{proposition}
Along the normal flow $\partial_t F=-u\nu$, the weighted mean curvature satisfies the evolution identity
\[
\partial_t H_\phi=\lap_\phi u+u\big(|A|^2+\Ricphi(\nu,\nu)\big). 
\]
\end{proposition}

\begin{proof}
By definition, split the time derivative:
\[
\partial_t H_\phi=\partial_t H-\partial_t\big(\inner{\overline{\grad}\phi}{\nu}\big). 
\]

For a normal flow $\partial_t F=-u\nu$ in a Riemannian manifold, the standard evolution law for mean curvature reads
\begin{equation}
\begin{aligned}
\partial_t H=\lap u+u\big(|A|^2+\Ric(\nu,\nu)\big). 
\end{aligned}
\label{1}
\end{equation}

Since $\inner{\nu}{\nu}\equiv1$, we have $\inner{\partial_t\nu}{\nu}=0$, hence $\partial_t\nu$ is tangent to $\Sigma_t$. 
For any tangent vector field $\tau\in T\Sigma_t$,
\[
\partial_t\inner{\nu}{\tau}=\inner{\partial_t\nu}{\tau}+\inner{\nu}{\overline{\nabla}_\tau\partial_t F}=0. 
\]
Substitute $\partial_t F=-u\nu$:
\[
\inner{\partial_t\nu}{\tau}-\inner{\nu}{\overline{\nabla}_\tau(u\nu)}
=\inner{\partial_t\nu}{\tau}-\grad_\tau u=0. 
\]
This holds for all tangent $\tau$, which yields
\begin{equation}
\begin{aligned}
\partial_t\nu=\grad u. 
\end{aligned}
\label{2}
\end{equation}

Apply Leibniz rule:
\[
\partial_t\inner{\overline{\grad}\phi}{\nu}
=\inner{\partial_t(\overline{\grad}\phi)}{\nu}+\inner{\overline{\grad}\phi}{\partial_t\nu}. 
\]
The function $\phi$ is fixed on the ambient space; only the evaluation point $F_t$ moves. Thus
\[
\partial_t(\overline{\grad}\phi)=\overline{\nabla}_{\partial_t F}\overline{\grad}\phi=\overline{\grad}^2\phi(\partial_t F,\,\cdot)=-u\,\overline{\grad}^2\phi(\nu,\,\cdot). 
\]
Taking pairing with $\nu$:
\[
\inner{\partial_t(\overline{\grad}\phi)}{\nu}=-u\,\overline{\grad}^2\phi(\nu,\nu). 
\]
Using (\ref{2}):
\[
\inner{\overline{\grad}\phi}{\partial_t\nu}=\inner{\overline{\grad}\phi}{\grad u}. 
\]
Combine the two terms:
\begin{equation}
\begin{aligned}
\partial_t\inner{\overline{\grad}\phi}{\nu}
=-u\,\overline{\grad}^2\phi(\nu,\nu)+\inner{\grad\phi}{\grad u}. 
\end{aligned}
\label{3}
\end{equation}

Substitute (\ref{1}) and (\ref{3}) into the split expression of $\partial_t H_\phi$:
\[
\begin{aligned}
\partial_t H_\phi
&=\big(\lap u+u|A|^2+u\Ric(\nu,\nu)\big)
-\Big(-u\,\overline{\grad}^2\phi(\nu,\nu)+\inner{\grad\phi}{\grad u}\Big)\\
&=\lap u-\inner{\grad\phi}{\grad u}
+u|A|^2+u\,\overline{\grad}^2\phi(\nu,\nu)+u\Ric(\nu,\nu). 
\end{aligned}
\]
By definition of the hypersurface weighted Laplacian,
\[
\lap u-\inner{\grad\phi}{\grad u}=\lap_\phi u. 
\]
Recall the Bakry--Émery Ricci tensor $\Ricphi(\nu,\nu)=\Ric(\nu,\nu)+\overline{\grad}^2\phi(\nu,\nu)$. 
Therefore
\[
\partial_t H_\phi=\lap_\phi u+u\big(|A|^2+\Ricphi(\nu,\nu)\big),
\]
which completes the proof. 
\end{proof}

\begin{lemma}\label{lem:3. 1}
Assume that \(\phi(r)\) and \(\lambda(r)\) satisfy (C1)-(C4) or (C1') and (C2)-(C5). Then \(H_\phi/V\) satisfies the pointwise
evolution inequality
\[
\partial_t\left(\frac{H_\phi}{V}\right)
\geq |A|^2-\phi'(r)\frac{\lambda'(r)}{\lambda(r)}. 
\]
Consequently, if
\[
H_\phi(\cdot,0)>0,
\]
then
\[
H_\phi(\cdot,t)>0
\]
for every time for which the flow exists. Hence \(V/H_\phi\) is
well-defined along the flow and satisfies
\[
\partial_t\left(\frac{V}{H_\phi}\right)
\leq -\lambda 'V. 
\]
\end{lemma}

\begin{proof}
We use the decomposition relating ambient and hypersurface weighted Laplacians:
\[
\lap_\phi V = \overline{\lap}_\phi V - \overline{\grad}^2 V(\nu,\nu) - H_\phi\partial_\nu V. 
\]
Insert this into the evolution law for $H_\phi$:
\[
\partial_t H_\phi = \overline{\lap}_\phi V - \overline{\grad}^2 V(\nu,\nu) - H_\phi\partial_\nu V + V|\big(|A|^2+\Ricphi(\nu,\nu)\big). 
\]
So
\[
\partial_t\left(\frac{H_\phi}{V}\right) = \frac{V\partial_t H_\phi - H_\phi\partial_t V}{V^2}. 
\]

By the chain rule
$$
\begin{aligned}
\partial_t V
=\partial_t\big(V\circ F_t\big)
=\big\langle \overline{\nabla}V,\partial_t F_t\big\rangle
=\langle \overline{\nabla}V,-V\nu\rangle
=-V\,\partial_\nu V,
\end{aligned}
$$

Substitute $\partial_t H_\phi$ and $\partial_t V = -V\partial_\nu V$ into the numerator:
$$
\begin{aligned}
V\partial_t H_\phi - H_\phi\partial_t V
&= V\big(\overline{\lap}_\phi V - \overline{\nabla}^2 V(\nu,\nu) - H_\phi\partial_\nu V + V|A|^2 + V\overline{\nabla}^2\phi(\nu,\nu)\big) + H_\phi V \partial_\nu V \\
&= V\big(\overline{\lap}_\phi V - \overline{\nabla}^2 V(\nu,\nu) + V|A|^2 + V\Ricphi(\nu,\nu)\big). 
\end{aligned}
$$

Substituting gives
\[
\partial_t\left(\frac{H_\phi}{V}\right)
=\frac{\overline{\lap}_\phi V - \overline{\nabla}^2 V(\nu,\nu)
+ V|A|^2 + V\Ricphi(\nu,\nu)}{V}. 
\]

The Lemma \ref{lem2. 4} implies that

\[
\partial_t\left(\frac{H_\phi}{V}\right)
\geq |A|^2-\phi'(r)\frac{\lambda'(r)}{\lambda(r)}. 
\]
Moreover, by (\ref{lambda}) and \ref{phi},
\[
-\frac{\phi'(r)}{r}\geq0. 
\]
Hence
\[
\partial_t\left(\frac{H_\phi}{V}\right)\geq0. 
\]
Since \(V>0\), the initial condition \(H_\phi(\cdot,0)>0\) implies
\[
\frac{H_\phi}{V}(\cdot,t)>0,
\]
and consequently
\[
H_\phi(\cdot,t)>0
\]
for all times under consideration. 

We may therefore differentiate the reciprocal. Namely,
\[
\partial_t\left(\frac{V}{H_\phi}\right)
=
\partial_t\left(\frac{H_\phi}{V}\right)^{-1}
=
-\frac{V^2}{H_\phi^2}
\partial_t\left(\frac{H_\phi}{V}\right). 
\]
Using the preceding evolution inequality yields
\[
\partial_t\left(\frac{V}{H_\phi}\right)
\leq
-\frac{V^2}{H_\phi^2}
\left(|A|^2-\phi'(r)\frac{\lambda'(r)}{\lambda(r)}\right). 
\]
Applying Lemma \ref{lem:weighted-mean-curvature-estimate}, we obtain
\[
\partial_t\left(\frac{V}{H_\phi}\right)
\leq -\lambda'V,
\]
which completes the proof. 
\end{proof}

\section{Proofs of Theorems 1.1—1.4}\label{subsec:4}
In this section we establish the main results of the paper. We begin with the proof of Theorem \ref{thm1. 2}.

\begin{theorem}\label{thm:4.1}
Assume that \(\phi(r)\) and \(\lambda(r)\) satisfy (C1') and (C2)-(C5). 
Let  \(\Sigma\subset N\times(0,\bar r)\) be a smooth closed
hypersurface satisfying $H_\phi>0$.
Then
\begin{equation}\label{hk11}
\int_\Sigma\frac{m\lambda'(r)-\phi'(r)\lambda(r)}{H_\phi}\,d\sigma_\phi
\ge
\int_\Omega [(m+1)\lambda'(r)-\phi'(r)\lambda(r)]\,d\mu_\phi. 
\end{equation}
Furthermore, equality in \eqref{hk11} holds if and only if one of the following alternatives holds:
\begin{enumerate}
\item $\Sigma$ is a coordinate slice $N\times\{r_0\}$ for some $r_0\in (0, \bar r)$; or
\item $\Sigma$ is a totally umbilical hypersurface contained in the region where $\phi$ is constant. 
\end{enumerate}
\end{theorem}

\begin{proof}

Since \(\overline{\Omega}\) is compact, there exist constants \(0<\underline V\leq
\overline V<\infty\) such that
\[
\underline V\leq V\leq \overline V
\qquad\text{on }\overline{\Omega}. 
\]
Consider the conformal metric
\[
\widehat g:=V^{-2}\bar{g}
\]
on \(\overline{\Omega}\). 
The above bounds imply that \(\widehat g\) and \(\bar{g}\) are uniformly
equivalent:
\[
\overline V^{-2}\bar{g}\leq \widehat g\leq \underline V^{-2}\bar{g}. 
\]
In particular, \((\overline{\Omega},\widehat g)\) is a compact
Riemannian manifold with smooth boundary. 

Since \(\widehat g=V^{-2}\bar{g}\), the inward
\(\widehat g\)-unit normal is given by
\[
\widehat\nu=-V\nu. 
\]

We consider the inward \(\widehat g\)-normal exponential map
\[
F:\mathcal D\longrightarrow \overline{\Omega},
\qquad
F(p,t):=\exp_p^{\widehat g}\bigl(t\widehat\nu(p)\bigr),
\]
where
\[
\mathcal D
:=
\bigl\{(p,t)\in\Sigma\times[0,\infty):0\leq t<\tau(p)\bigr\},
\]
and \(\tau(p)\) is the inward cut time of the
\(\widehat g\)-normal geodesic starting from \(p\). 

For \(0\leq t<\tau(p)\), the curves
\[
t\longmapsto F(p,t)
\]
are \(\widehat g\)-unit speed geodesics. Since the conformal change of
metric preserves angles, these geodesics remain orthogonal to the
hypersurfaces
\[
\Sigma_t
:=
F\bigl(\{p\in\Sigma:\tau(p)>t\},t\bigr). 
\]

Let
\[
\operatorname{Cut}_{\widehat g}(\Sigma)
:=
\left\{
F\bigl(p,\tau(p)\bigr):
\tau(p)<\infty
\right\}
\]
be the  \(\widehat g\)-cut locus of \(\Sigma\). By the classical
cut locus theorem for the distance function to a smooth hypersurface,
\[
\operatorname{Vol}_{\widehat g}
\bigl(\operatorname{Cut}_{\widehat g}(\Sigma)\bigr)=0,
\]
and
\[
\Omega\setminus\operatorname{Cut}_{\widehat g}(\Sigma)
=
F(\mathcal D). 
\]
Since
\[
d\mu_{\widehat g}=V^{-(m+1)}\,d\mu
\]
and the measures
\(d\mu_{\widehat g}\), \(d\mu\), and \(d\mu_\phi\) are mutually
absolutely continuous on \(\overline{\Omega}\). Consequently,
\[
\mu_\phi\bigl(\operatorname{Cut}_{\widehat g}(\Sigma)\bigr)=0. 
\]
Thus, the speed-\(V\) flow sweeps out all of \(\Omega\) up to a set of
zero weighted measure. 

Let $J_\phi(x,t)$ be the weighted tangential Jacobian of the flow,
defined by
\begin{equation}
F_t^*(d\sigma_\phi(t))
=
J_\phi(x,t)d\sigma_\phi(x),
\qquad
F_t(x):=F(x,t). 
\label{eq:weighted-Jacobian}
\end{equation}

Since the normal speed is $-V\nu$, the weighted area element evolves
according to
\begin{equation}
\partial_t d\sigma_\phi(t)
=
-VH_\phi\,d\sigma_\phi(t). 
\label{eq:weighted-area-evolution}
\end{equation}
Consequently,
\begin{equation}
\partial_tJ_\phi
=
-VH_\phi J_\phi. 
\label{eq:weighted-Jacobian-evolution}
\end{equation}

Now define
\[
h(x,t):=
\frac{V(F(x,t))}{H_\phi(F(x,t),t)}. 
\]
By Lemma \ref{lem:3. 1}, we have
\begin{equation}
\partial_th
=
\partial_t\left(\frac{V}{H_\phi}\right)
\leq
-\lambda'V. 
\label{eq:h-evolution}
\end{equation}
Combining \eqref{eq:weighted-Jacobian-evolution} and
\eqref{eq:h-evolution}, we obtain
\[
\begin{aligned}
\partial_t(hJ_\phi)
&=
(\partial_th)J_\phi+h\partial_tJ_\phi
\\
&\leq
-\lambda'VJ_\phi
-\frac{V}{H_\phi}VH_\phi J_\phi
\\
&=
-\bigl(\lambda'V+V^2\bigr)J_\phi. 
\end{aligned}
\]
It follows that
\begin{equation}
\partial_t(hJ_\phi)
\leq
-VWJ_\phi. 
\label{eq:key-Jacobian-inequality}
\end{equation}

Integrating \eqref{eq:key-Jacobian-inequality} from $0$ to
$\tau(x)$ gives
\begin{equation}
\frac{V(x)}{H_\phi(x)}
\geq
\int_0^{\tau(x)}
V(F(x,t))W(F(x,t))J_\phi(x,t)\,dt
+
\liminf_{t\nearrow\tau(x)}
\left(
\frac{V}{H_\phi}J_\phi
\right)(x,t). 
\label{eq:pointwise-HK}
\end{equation}

We next apply the weighted coarea formula to the map
\[
F:\mathcal D\longrightarrow\Omega. 
\]
Since
\[
|\partial_tF|_{\bar g}=V,
\]
the  weighted Jacobian of $F$ is
$VJ_\phi$. 
Therefore, for every nonnegative measurable function $f$ on $\Omega$,
\begin{equation}
\int_\Omega f\,d\mu_\phi
=
\int_\Sigma
\int_0^{\tau(x)}
f(F(x,t))V(F(x,t))J_\phi(x,t)
\,dt\,d\sigma_\phi(x). 
\label{eq:weighted-coarea}
\end{equation}
Taking $f=W$ in \eqref{eq:weighted-coarea}, we obtain
\begin{equation}
\int_\Omega W\,d\mu_\phi
=
\int_\Sigma
\int_0^{\tau(x)}
VWJ_\phi\,dt\,d\sigma_\phi. 
\label{eq:coarea-W}
\end{equation}

In the case (C1'), $M$ is smooth at the origin given by $\{r=0\}$. So, integrating \eqref{eq:pointwise-HK} over $\Sigma$ and applying
\eqref{eq:coarea-W}, we get

\begin{equation}
\begin{aligned}
\int_\Sigma\frac{V}{H_\phi}\,d\sigma_\phi
\geq
\int_\Omega W\,d\mu_\phi. 
\end{aligned}
\label{eq:HK-before-N0}
\end{equation}

\textbf{The equality case:}
Assume that equality holds in
\[
\int_\Sigma\frac{V_\phi}{H_\phi}\,d\sigma_\phi
=
\int_\Omega W_\phi\,d\mu_\phi. 
\]
Then equality must hold in every step of the proof. In particular,
equality holds in the pointwise evolution inequality of Lemma
\ref{lem:3. 1}. By the equality condition of the Cauchy--Schwarz inequality in Lemma \ref{lem:weighted-mean-curvature-estimate}, $\Sigma$ is totally umbilical,

and

\begin{equation}
\frac{\phi'}{\lambda\lambda'}\inner{X}{\nu}^2=\frac{\phi'}{\lambda \lambda'}|X|^2_{\bar{g}}
\qquad\text{on }\Sigma. 
\label{eq:equality-phi}
\end{equation}

We now distinguish two cases. 

\medskip

\noindent\textbf{Case 1: $\phi'(r)\neq0$ at some point of $\Sigma$. }

Since $\phi'\leq0$, equation \eqref{eq:equality-phi} implies that
\[
|\nabla r|_{\bar{g}}=0
\]
at every point where $\phi'(r)<0$. Hence
\[
\partial_r\perp T\Sigma,
\]
or equivalently,
\[
\nu=\partial_r
\]
on the open set
\[
\Sigma_-:=\{x\in\Sigma:\phi'(r(x))<0\}. 
\]

Since $\Sigma$ is connected and $\Sigma_-$ is nonempty, it follows
that $r$ is constant on $\Sigma$. Therefore there exists $r_0\in
(0,\bar r)$ such that
\[
\Sigma=N\times\{r_0\}. 
\]
That is, $\Sigma$ is a coordinate slice. 

\medskip

\noindent\textbf{Case 2: $\phi'(r)=0$ on $\Sigma$. }

In this case, $\Sigma$ is contained in the region where $\phi'(r)= 0$. Moreover, $\Sigma$ is a totally umbilical hypersurface in this region. 

Therefore, equality implies one of the following alternatives:
\begin{enumerate}
\item $\Sigma$ is a coordinate slice $N\times\{r_0\}$; or
\item $\Sigma$ is a totally umbilical hypersurface contained in the region where $\phi$ is constant. 
\end{enumerate}
\end{proof}

We now turn to the proof of Theorem \ref{thm1. 2'}.
\begin{theorem}\label{thm4.2}
Assume that \(\phi(r)\) and \(\lambda(r)\) satisfy (C1)-(C4). 
Let  \(\Sigma\subset N\times(0,\bar r)\) be a smooth closed
hypersurface satisfying $H_\phi>0$.
\begin{enumerate}
\item If $\Sigma$ is null-homologous, then \eqref{hk11} holds.
\item If $\Sigma$ is homologous to $N_0=N\times \{0\}$,
then
\begin{equation}\label{hk22}
\int_\Sigma\frac{m\lambda'(r)-\phi'(r)\lambda(r)}{H_\phi}\,d\sigma_\phi
\ge
\int_\Omega [(m+1)\lambda'(r)-\phi'(r)\lambda(r)]\,d\mu_\phi
+\lambda(0)\int_{N_0 \cap \partial \Omega}d\sigma_\phi. 
\end{equation}
Furthermore, equality in \eqref{hk22} holds if and only if one of the following alternatives holds:
\begin{enumerate}
\item $\Sigma$ is a coordinate slice $N\times\{r_0\}$; or
\item $\Sigma$ is a totally umbilical hypersurface contained in the region where $\phi$ is constant. 
\end{enumerate}
\end{enumerate}
\end{theorem}

\begin{proof}
In the case (C1), $M$ is a manifold with boundary $\partial M=N\times \{0\}$. If $\Sigma$ is null-homologous, The proof proceeds in exactly the same way as that of Theorem \ref{thm:4.1}.

If $\Sigma$ is null-homologous. Integrating \eqref{eq:pointwise-HK} over $\Sigma$ and applying
\eqref{eq:coarea-W}, we get

\begin{equation}
\begin{aligned}
\int_\Sigma\frac{V}{H_\phi}\,d\sigma_\phi
\geq
\int_\Omega W\,d\mu_\phi+\int_\Sigma
\liminf_{t\nearrow\tau(x)}
\left(
\frac{V}{H_\phi}J_\phi
\right)(x,t)
\,d\sigma_\phi(x). 
\end{aligned}
\label{eq:HK-before-N00}
\end{equation}

It remains to estimate the terminal term. Along every coordinate
slice
\[
N_r=N\times\{r\},
\]
with unit normal $\partial_r$, we have
\[
H=m\frac{\lambda'}{\lambda}. 
\]
Hence,
\[
H_\phi
=
m\frac{\lambda'}{\lambda}-\phi'
=
\frac{m\lambda'-\phi'\lambda}{\lambda}
=
\frac{V}{\lambda}. 
\]
Therefore,
\begin{equation}
\frac{V}{H_\phi}=\lambda
\qquad\text{on }N_r. 
\label{eq:V-Hphi-slice}
\end{equation}

The terminal set of the inward normal flow consists of the inner
boundary $N_0$ and the cut locus. The cut locus contribution is
nonnegative. On the part of the terminal set lying on $N_0$, by
\eqref{eq:V-Hphi-slice} and the area formula, we have
\begin{equation}
\int_\Sigma
\liminf_{t\nearrow\tau(x)}
\left(
\frac{V}{H_\phi}J_\phi
\right)(x,t)
\,d\sigma_\phi(x)
\geq
\lambda(0)\int_{N_0 \cap \partial \Omega}d\sigma_\phi. 
\label{eq:N0-contribution0}
\end{equation}

Combining \eqref{eq:HK-before-N00} and \eqref{eq:N0-contribution0}, we
conclude that
\[
\int_\Sigma\frac{V}{H_\phi}\,d\sigma_\phi
\geq
\int_\Omega W\,d\mu_\phi
+
\lambda(0)\int_{N_0 \cap \partial \Omega}d\sigma_\phi. 
\]

we obtain
\[
\int_\Sigma\frac{V_\phi}{H_\phi}\,d\sigma_\phi
\geq
\int_\Omega W_\phi\,d\mu_\phi
+
\lambda(0)\int_{N_0}d\sigma_\phi. 
\]
The characterization of the equality case follows in the same way as in Theorem \ref{thm:4.1}.
\end{proof}

\begin{theorem}
Assume that \(\phi(r)\) and \(\lambda(r)\) satisfy (C1)-(C4) or (C1') and (C2)-(C5).  Let $\Sigma\subset M^{m+1}$ be a smooth closed embedded $\phi$-CMC hypersurface, i. e. , $H_\phi$ is constant on $\Sigma$. Then one of the following alternatives holds:
\begin{itemize}
\item[(1)] $\Sigma$ is a coordinate slice $N\times \{r_0\}$; or
\item[(2)] $\Sigma$ is a totally umbilical hypersurface contained in the region where $\phi$ is constant. 
\end{itemize}
\end{theorem}

\begin{proof}
By Lemma \ref{lem:weighted-divergence-X},Lemma \ref{lem:weighted-Minkowski} and weighted divergence theorem,we have

\[
\int_\Sigma V\,d\sigma_\phi
=\int_\Sigma H_\phi\langle X,\nu\rangle\,d\sigma_\phi=H_\phi\int_{\Omega} W \, d\mu_{\phi}+H_\phi\int_{N_0 \cap \partial \Omega} \lambda(0)d\sigma_\phi. 
\]
Under the conditions (C1') or that $\Sigma$ is null‑homologous, the second term on the right‑hand side of the above identity vanishes. Since $V,W>0$ and $\lambda(0)\ge 0$, it follows that $H_\phi>0$. 
Furthermore, the above identity forces equality in Theorem \ref{thm:4.1} and Theorem \ref{thm4.2}, which completes the proof. 
\end{proof}

\begin{proof}[Proof of Theorem \ref{thm1. 1}]
We apply Theorem \ref{thm:4.1} to the Euclidean space written in polar
coordinates,
\[
(\mathbb R^{m+1},\bar g)
=
\bigl([0,+\infty)\times\mathbb S^m,\,
dr^2+r^2g_{\mathbb S^m}\bigr). 
\]
Thus, the warping function is
\[
\lambda(r)=r,\qquad \kappa_N =1. 
\]
We take the weight function
\[
\phi(r)=-\frac{r^2}{4}. 
\]

Then
\[
\lambda'=1,\qquad \lambda''=0,\qquad
\phi'=-\frac r2. 
\]
Consequently,
\[
V=m\lambda'-\phi'\lambda=m+\frac{r^2}{2},
\]
and
\[
W=(m+1)\lambda'-\lambda\phi'
=m+1+\frac{r^2}{2}. 
\]
Moreover,
\[
\mathcal T_{\lambda,\kappa_N}
=
\frac{(m-1)-r\cdot0-(m-1)}{r^2}=0. 
\]
Since $V'=r$ and $V''=1$, it follows that
\[
\begin{aligned}
\mathcal C_{\lambda,\phi}
&=V''+\bigl((m-1)q-\phi'\bigr)V'
+\mathcal T_{\lambda,\kappa_N}V+2\phi' qV\\
&=1+\left(\frac{m-1}{r}+\frac r2\right)r
+2\left(-\frac {r}{2}\right)\frac{1}{r}\left(m+\frac{r^2}{2}\right)\\
&=0. 
\end{aligned}
\]
Hence the result follows immediately from Theorem \ref{thm:4.1}. 
\end{proof}

\begin{appendix}\label{appendix}

\section{Weighted Heintze--Karcher Inequalities in Space Forms}

The previous section established a  weighted Heintze–Karcher inequality for hypersurfaces in weighted warped product manifolds. In this section, we specialize this abstract result to the classical simply‑connected space forms of constant sectional curvature. Working within geodesic polar coordinates centered at a fixed origin, our goal is to derive concrete, explicit versions of the weighted Heintze–Karcher inequality for hypersurfaces in these model geometries.

Let $\mathbb M_{\kappa}^{m+1}$ be the simply connected \((m+1)\)-dimensional space form of constant
sectional curvature
\[
\kappa\in\{-1,0,1\}.
\]
Thus,
\[
\mathbb M_{-1}^{m+1}=\mathbb H^{m+1},
\qquad
\mathbb M_{0}^{m+1}=\mathbb R^{m+1},
\qquad
\mathbb M_{1}^{m+1}=\mathbb S^{m+1}.
\]

Fix a point
\[
o\in\mathbb M_{\kappa}^{m+1},
\]
and define
\[
r(x):=\operatorname{dist}_{\mathbb M_\kappa}(o,x).
\]
In the spherical case \(\kappa=1\), we assume that
$$\Omega \subset \mathbb{S}^n_+.
$$

Define the generalized sine and cosine functions by
\[
\lambda(r)=s_\kappa(r):=
\begin{cases}
\sinh r,&\kappa=-1,\\[0.15cm]
r,&\kappa=0,\\[0.15cm]
\sin r,&\kappa=1,
\end{cases}
\]
and
\[
\lambda'(r)=c_\kappa(r):=
s_\kappa'(r)=
\begin{cases}
\cosh r,&\kappa=-1,\\[0.15cm]
1,&\kappa=0,\\[0.15cm]
\cos r,&\kappa=1.
\end{cases}
\]
They satisfy
\[
s_\kappa''+\kappa s_\kappa=0,
\qquad
c_\kappa'=-\kappa s_\kappa,
\qquad
c_\kappa^2+\kappa s_\kappa^2=1.
\]
Since $\kappa_N=1$, we obtain
\[
\mathcal T_{\lambda,1}=\frac{(m-1)-\lambda\lambda''
-(m-1)(\lambda')^2}{\lambda^2}=m\kappa.
\]
\textbf{Case 1: }For \(\kappa=-1\), namely \(\mathbb H^{m+1}\),
\[
\lambda(r)=\sinh r,
\qquad
\lambda'(r)=\cosh r.
\]
Therefore,
\[
V=m\cosh r-\phi'(r)\sinh r,
\]
and
\[
C_{\lambda,\phi}
=\bigl(-\phi'-\phi'''+\phi'\phi''\bigr)\sinh r -
\bigl((m+1)\phi''+(\phi')^2\bigr)\cosh r
+(m+1)\phi'\frac{\cosh^2 r}{\sinh r}.
\]
If $C_{\lambda,\phi} \ge 0$ and $H_\phi>0$, we have
\[
\int_\Sigma
\frac{
m\cosh r-\phi'(r)\sinh r
}{
H_\phi
}
\,d\sigma_\phi
\geq
\int_\Omega
\left[
(m+1)\cosh r-\phi'(r)\sinh r
\right]
\,d\mu_\phi.
\]
\textbf{Case 2: }For \(\kappa=0\), namely \(\mathbb R^{m+1}\),
\[
\lambda(r)=r,
\qquad
\lambda'(r)=1.
\]
Therefore
\[
V=m-r\phi'(r),
\]
and
\[
C_{\lambda,\phi}
=-r\phi'''
-\bigl((m+1)\phi''+(\phi')^2\bigr)
+\frac{(m+1)\phi'}{r}
+r\phi'\phi''.
\]
If $C_{\lambda,\phi} \ge 0$ and $H_\phi>0$, we have
\[
\int_\Sigma
\frac{
m-r\phi'(r)
}{
H_\phi
}
\,d\sigma_\phi
\geq
\int_\Omega
\left[
m+1-r\phi'(r)
\right]
\,d\mu_\phi.
\]
\textbf{Case 3: }For \(\kappa=1\), namely \(\mathbb S^{m+1}\), under the hemisphere
assumption
\[
\Omega\subset B_R(o),
\qquad
R<\frac{\pi}{2},
\]
we have
\[
\lambda(r)=\sin r,
\qquad
\lambda'(r)=\cos r.
\]
Hence
\[
V=m\cos r-\phi'(r)\sin r,
\]
and 
\[
C_{\lambda,\phi}=
\bigl(\phi'-\phi'''+\phi'\phi''\bigr)\sin r -
\bigl((m+1)\phi''+(\phi')^2\bigr)\cos r
+(m+1)\phi'\frac{\cos^2 r}{\sin r}.
\]
If $C_{\lambda,\phi} \ge 0$ and $H_\phi>0$, we have
\[
\int_\Sigma
\frac{
m\cos r-\phi'(r)\sin r
}{
H_\phi
}
\,d\sigma_\phi
\geq
\int_\Omega
\left[
(m+1)\cos r-\phi'(r)\sin r
\right]
\,d\mu_\phi.
\]

\subsection{Some weights satisfying the structural conditions}

The following table lists several explicit radial potentials for which
$C_{\lambda,\phi}\geq 0$.

\[
\renewcommand{\arraystretch}{1.8}
\begin{array}{|c|c|c|c|}
\hline
\text{Ambient space} & \lambda(r) & \phi(r) & C_{\lambda,\phi} 

\\ \hline
\mathbb R^{m+1}
&
r
&
\displaystyle -r^k
&
k(k-2)r^{k-2}
\left(m+k+kr^k\right),k \ge 2

\\ \hline
\mathbb S_+^{m+1}
&
\sin r
&
-e^{-\sqrt2\cos r}
&
\begin{aligned}
\sqrt{2}\,e^{-\sqrt{2}\cos r}\sin^2 r
\biggl[
-2+\sqrt{2}(m+4)\cos r \\
\qquad +2\sin^2 r\biggl(1+e^{-\sqrt{2}\cos r}\biggr)
\biggr]
\end{aligned}

\\ \hline
\mathbb H^{m+1}
&
\sinh r
&
- \cosh^k r
&
\begin{aligned}
{}&
k\sinh^2r\,\cosh^{k-3}r
\Bigl[
\bigl(2+(m+4)(k-1)\bigr)\cosh^2r
\\
&\qquad\qquad
+(k-1)\sinh^2r
\bigl(k-2+k\cosh^k r\bigr)
\Bigr], k\ge 1
\end{aligned}

\\ \hline
\end{array}
\]

\end{appendix}

\bibliographystyle{plain}
\bibliography{References. bib}

\end{document}